\documentclass[12pt]{amsart}
\usepackage{color}
\usepackage{amssymb}
\usepackage{graphicx}
 \usepackage{epstopdf}
 \usepackage[margin=1.2in]{geometry}
 \usepackage{hyperref}

\newtheorem{theorem}{Theorem}[section]
\newtheorem{lemma}[theorem]{Lemma}
\newtheorem{proposition}[theorem]{Proposition}

\theoremstyle{definition}

\newtheorem{example}[theorem]{Example}

\theoremstyle{remark}

\numberwithin{equation}{section}

\begin{document}
\title[EQUATION
WITH PRODUCT OF ITERATES]{DIFFERENTIABLE SOLUTIONS OF AN  EQUATION \\
	\vspace{0.05cm} 
	WITH PRODUCT OF ITERATES}

\author{CHAITANYA GOPALAKRISHNA}
\address{ Statistics and Mathematics Unit, Indian Statistical Institute,  R.V. College Post, Bangalore-560059, India}
%\curraddr{Department of Mathematics and Statistics,
%Case Western Reserve University, Cleveland, Ohio 43403}
\email{cberbalaje@gmail.com, chaitanya\_vs@isibang.ac.in}
\thanks{The author is supported by Indian Statistical Institute,
	Bangalore in the form of a Visiting Scientist position through the J. C. Bose Fellowship
	of Prof. B. V. Rajarama Bhat.
}

%    General info
\subjclass[2000]{Primary 39B12; Secondary 47J05.}

%\date{May ??, 2021.}
%\today

\keywords{Functional equation, iteration, nonlinear combination, contraction principle.}

\begin{abstract}
In the previous work \cite{CG-MV-SW-WZ}, we considered continuous solutions of an iterative equation involving the multiplication of iterates. In this paper, we continue to investigate this equation for differentiable solutions. Similar to continuous solutions until \cite{CG-MV-SW-WZ}, there is no obtained result on differentiable solutions of such an equation on non-compact intervals of  $\mathbb{R}$. Although our strategy here is to use conjugation to reduce the equation to the well-known polynomial-like iterative equation as in \cite{CG-MV-SW-WZ}, all known results on differentiable solutions of the latter are given on compact intervals. We re-explore polynomial-like iterative equation on the whole of R and prove the existence and uniqueness of differentiable solutions of our equation on $\mathbb{R}_+$ and $\mathbb{R}_-$.

%Using the fixed point theorem of Banach, we discuss the existence and  uniqueness of differentiable solutions of an iterative equation involving product of iterates. 
\end{abstract}

\maketitle

%%%%%%%%%%%%%%%%%%%%%%%%%%%%%%%%%%%%%%%%%%%%%%%%%%%%%%%%%%%%%%%%%%%%%%%%
%%%%%%%%%%%%%%%%%%%%%%%%%%%%%%%%%%%%%%%%%%%%%%%%%%%%%%%%%%%%%%%%%%%%%%%%
\section{Introduction}
The $n$-th order iterates of a map $f: X\to X$ on a nonempty set $X$, denoted by $f^n$,
are defined recursively by $f^0={\rm id}$, the identity map, and $f^{n+1}=f\circ f^n$.
The iteration operation, being an essential operation having applications to numerical computations and computer loop programs, is being investigated for its many interesting and complicated properties and, in particular
a lot of attention (see \cite{Baron-Jarczyk,Kuczma} for example) is paid to the so-called 
{\it iterative equations}, which have the general form
\begin{eqnarray}\label{phi}
\Phi(f(x), f^2(x),..., f^n(x))=F(x),
\end{eqnarray}
%$$
%\Phi(f(x), f^2(x),..., f^n(x))=F(x),
%$$
where $F$ and $\Phi$ are given, and $f$ is unknown.
Some special cases of this equation,
for example,
iterative root problem (\cite{Kuczma1968,Targonski}), which is a special case of the invariant curve problem \cite{Kuczma}, and dynamics of a quadratic map (\cite{greenfield})
%and invariant curve problem \cite{Kuczma}
%Feigenbaum's equation related to period doubling bifurcations (\cite{Mc}),
are interesting topics in dynamical systems.

Although there are plentiful  results (see  \cite{Mu-Su, Si,Wang-Si} for example) on the solutions of \eqref{phi} when $\Phi$ is a Lipschitzian,
 the basic form
\begin{eqnarray}\label{(*)}
\lambda_1f(x)+\lambda_2f^2(x)+\cdots + \lambda_n f^n(x)= F(x)
\end{eqnarray}
with $\Phi$ in a linear combination, called the {\it polynomial-like iterative equation}, is still being considered
for deeper investigation.
Continuous solutions, differentiable solutions, convex solutions and decreasing solutions, and equivariant solutions
of \eqref{(*)} are discussed in \cite{Xu-Zhang,Bing-Weinian,zhang1990,  Zhang-Edinb}, respectively.
It is also interesting to discuss $\Phi$ of nonlinear combination.
In 2007 Zdun and Zhang \cite{Zdun-Zhang} discussed \eqref{phi}
%$$
%\Phi(f(x), f^2(x),..., f^n(x))=F(x),
%$$
for continuous solutions on the compact space $S^1$, the unit circle in $\mathbb{C}$, and recently Gopalakrishna et al. \cite{CG-MV-SW-WZ} 
investigated \eqref{phi}
in the special case
$\Phi(u_1, u_2, \ldots, u_n)=\prod_{k=1}^{n}u_k^{\lambda_k}$,
i.e., an iterative equation involving product of iterates
\begin{eqnarray}
\begin{array}{ll}
(g(x))^{\lambda_1} (g^{2}(x))^{\lambda_2}  \cdots (g^n(x))^{\lambda_n} =G(x),
\end{array}
%\eqno(1)
\label{1}
\end{eqnarray}
 for continuous solutions
on the non-compact spaces $\mathbb{R}_+:=(0, \infty)$ and $\mathbb{R}_-:=(-\infty,0)$.

In this paper,  we continue to inverstige \eqref{1} considering its differentiable solutions.
Unlike those \cite{Mu-Su,Wang-Si,Si-Zhang,zhang1989,zhang1990} on compact intervals, 
our work to (\ref{1}) is focused on investigating \eqref{(*)} on the whole $\mathbb{R}$ as done in \cite{CG-MV-SW-WZ}.
Our approach is to restrict the discussion of \eqref{1} on $\mathbb{R}_+$ and use an exponential function
to reduce in conjugation to the well-known form of polynomial-like iterative equation (\ref{(*)}) on the whole $\mathbb{R}$. 
Note that all found results on \eqref{(*)}
are given on a compact interval, none of which
are applicable to our case.
In this paper, using Banach's contraction principle, we generally discuss \eqref{(*)} on the whole $\mathbb{R}$
and use obtained result to give   sufficient conditions for
existence and uniqueness of differentiable solutions for \eqref{1} on $\mathbb{R}_+$ and $\mathbb{R}_-$. The paper is organized as follows. In Section \ref{Sec2} we give preliminaries. In Section \ref{Sec3} we discuss the existence and uniqueness of differentiable solutions of \eqref{1}. Finally, in Section \ref{Sec4} we illustrate our result with an example and indicate some problems for future discussion. 

\section{Preliminaries}\label{Sec2}
Let
$\mathcal{C}_b(\mathbb{R}_+)$ (resp. $\mathcal{C}_b(\mathbb{R})$) denote the set of all bounded continuous self-maps of $\mathbb{R}_+$ (resp. $\mathbb{R}$), $\mathcal{C}^1(\mathbb{R}_+)$ (resp. $\mathcal{C}^1(\mathbb{R})$) the set of all continuously differentiable maps in $\mathcal{C}_b(\mathbb{R}_+)$ (resp. $\mathcal{C}_b(\mathbb{R})$),  and   $\mathcal{C}_b^1(\mathbb{R})$ the set of all  maps in  $\mathcal{C}^1(\mathbb{R})$ with bounded derivatives.
Then $\mathcal{C}_b(\mathbb{R})$
is a Banach space in the uniform norm 
$\|f\|_\infty:=\sup \{|f(x)|: x\in \mathbb{R}\},$ and $\mathcal{C}^1_b(\mathbb{R})$ is a normed linear space in the norm  $\|f \|_{\mathcal{C}^1}:=\|f\|_\infty +\|f'\|_\infty$, where $f'$ denotes the derivative of $f$. 
%\begin{eqnarray}
%\|f \|_{\mathcal{C}^1}:=\|f\|_\infty +\|f'\|_\infty.
%\end{eqnarray}

Consider $g$ on $\mathbb{R}_+$. We can use the exponential map $\psi(x)=e^x$ to conjugate $g$ to get a self-map $f(x):=\log g(e^x)$ on the whole $\mathbb{R}$
%(one-to-one if $g$ is one-to-one),
%ÕâÊ±$f$ÔÚ$0$µãÓÃ²»ÉÏ¡£µÈ°Ñ$F$Çó³öÀ´ºóÔÙ²¹³ä¶¨Òå£º
%\begin{eqnarray*}
%f(x):=
%\left\{
%\begin{array}{ll}
%\exp(F(\log x))   & \mbox{ for } x\in (0,+\infty);
%\\
%0                 &\mbox{ for } x=0.
%\end{array}
%\right.
%\end{eqnarray*}
and reduce  (\ref{1}) to the polynomial-like one \eqref{(*)} on $\mathbb{R}$,
%\begin{eqnarray}
%\lambda_1 f(x) +\lambda_2 f^2(x)+\cdots +\lambda_nf^n(x)= F(x), \quad    x\in \mathbb{R},
%\label{(*)}
%\end{eqnarray}
where $F(x):=\log G(e^x)$.

The following two propositions shows that it suffices to prove existence (resp. uniqueness) for \eqref{(*)} on the whole $\mathbb{R}$ in order to prove the existence (resp. uniqueness) of solution for \eqref{1} on $\mathbb{R}_+$ and $\mathbb{R}_-$. 
%So,  in what follows, we investigate $\mathcal{C}^1$-solutions of equation (\ref{(*)}) on $\mathbb{R}$.

\begin{proposition}\label{P3}
	A map $g$ is a solution (resp. unique solution) of \eqref{1} in $\mathcal{X}\subseteq \mathcal{C}^1(\mathbb{R}_+)$
	if and only if $f(x):=\psi^{-1}(g(\psi(x)))$
	is a solution (resp. unique solution) of \eqref{(*)} in $\mathcal{Y}\subseteq \mathcal{C}^1(\mathbb{R})$, where $\psi(x)=e^x$ and $\mathcal{Y}=\{\psi^{-1}\circ g\circ \psi: g\in \mathcal{X}\}$.
\end{proposition}

\begin{proof}
	Let $g$ be a solution of 
	\eqref{1} in $\mathcal{X}$.  Since $\psi$ is a diffeomorphism of $\mathbb{R}$ onto $\mathbb{R}_+$, clearly $f\in \mathcal{Y} \subseteq \mathcal{C}^1(\mathbb{R})$.  Also, 
	\begin{eqnarray*}
		\sum_{k=1}^{n}\lambda_kf^k(x)=\sum_{k=1}^{n}\lambda_k\log g^k(e^x)
		= \log\left(\prod_{k=1}^{n}(g^k(x))^{\lambda_k}\right) 
		=\log (G(e^x))=F(x)
	\end{eqnarray*}
	%	\begin{eqnarray*}
	%	\sum_{k=1}^{n}\lambda_kf^k(x)&=&\sum_{k=1}^{n}\lambda_k\log g^k(e^x)\\
	%	&=& \log\left(\prod_{k=1}^{n}(g^k(x))^{\lambda_k}\right) \\
	%	&=&\log (G(e^x))=F(x),
	%	\end{eqnarray*}
for all $x\in \mathbb{R}$,	implying that $f$ is a solution of \eqref{(*)} on $\mathbb{R}$.  The converse follows similarly. To prove the uniqueness, assume that \eqref{1} has a unique solution in $\mathcal{X}$ and suppose that $f_1, f_2\in \mathcal{Y}$ are any two solutions of \eqref{(*)}. Then, by ``if'' part of what we have proved above, there exist solutions  $g_1, g_2 \in \mathcal{X}$ of \eqref{1} such that $f_1=\psi^{-1}\circ g_1\circ \psi$ and $f_2=\psi^{-1}\circ g_2\circ \psi$. By our assumption, we have $g_1=g_2$ and therefore $f_1=f_2$. The proof of converse is similar.
\end{proof}

\begin{proposition}\label{P4}
	Let
	$\lambda_k\in \mathbb{Z}$ for $1\le k\le n$ such that $\sum_{k=1}^{n}\lambda_k$ is odd.
	Then a map $g$ is a solution (resp. unique solution) of \eqref{1} in $\mathcal{X}\subseteq \mathcal{C}^1(\mathbb{R}_-)$
	if and only if $h(x):=\psi^{-1}(g(\psi(x)))$
	is a solution (resp. unique solution) of the equation
	\begin{eqnarray}\label{15}
	\prod_{k=1}^{n}(h^{k}(x))^{\lambda_k}=H(x)
	\end{eqnarray}
	in  $\mathcal{Y}\subseteq \mathcal{C}^1(\mathbb{R}_+)$,
	where $\psi(x)=-x$, $H(x)=\psi^{-1}(G(\psi(x)))$ and $\mathcal{Y}=\{\psi^{-1}\circ g\circ \psi: g\in \mathcal{X}\}$.
\end{proposition}

\begin{proof}
	Let $g$ be a solution of \eqref{1} in $\mathcal{X}$.
	Since $\psi$ is a diffeomorphism of $\mathbb{R}_+$ onto $\mathbb{R}_-$, clearly $h\in \mathcal{Y}\subseteq \mathcal{C}^1(\mathbb{R}_+)$.
	Also, for each $x\in \mathbb{R}_+$ and $k\in \{1,2, \ldots, n\}$, we have $H(x)=-G(-x)$ and $h^k(x)=-g^k(-x)$. Therefore
	\begin{eqnarray*}
		\prod_{k=1}^{n}(h^{k}(x))^{\lambda_k}
		= \prod_{k=1}^{n}(-g^k(-x))^{\lambda_k}
		&=&(-1)^{\sum_{k=1}^{n}\lambda_k}\prod_{k=1}^{n} (g^k(-x))^{\lambda_k}\\
		&=&-\prod_{k=1}^{n} (g^k(-x))^{\lambda_k}
		=-G(-x)=H(x)
	\end{eqnarray*}
for each $x\in \mathbb{R}_+$, because
	$\sum_{k=1}^{n}\lambda_k$ is odd,
	implying that $h$ is a solution of \eqref{15} on $\mathbb{R}_+$.
	The converse follows similarly. Further, the proof of uniqueness is similar to that of Proposition \ref{P3}.
\end{proof}

Let $J=[c,d]$ and $I=[a,b]$ be compact intervals in $\mathbb{R}_+$ and $\mathbb{R}$, respectively with non-empty interiors. 
Let
$\mathcal{C}(J)$ (resp. $\mathcal{C}^1(J)$) be the set of all continuous (resp. continuously differentiable) self-maps of $J$. Similarly we define $\mathcal{C}(I)$ and $\mathcal{C}^1(I)$.
For each $g\in \mathcal{C}_b(\mathbb{R}_+)$ (resp. $f\in \mathcal{C}_b(\mathbb{R})$) and subinterval $J'$ of $\mathbb{R}_+$ (resp. $I'$ of $\mathbb{R}$),
let
$\|g\|_\infty^{J'}:=\sup \{|g(x)|: x\in J'\}$  (resp. $\|f\|_\infty^{I'}:=\sup \{|f(x)|: x\in I'\}$). Similarly, for $g\in \mathcal{C}^1(\mathbb{R}_+)$ (resp. $f\in \mathcal{C}_b^1(\mathbb{R})$), let $\|g\|_{\mathcal{C}^1}^{J'}:=\|g\|_\infty^{J'}+\|g'\|_\infty^{J'}$  (resp. $\|f\|_{\mathcal{C}^1}^{I'}:=\|f\|_\infty^{I'}+\|f'\|_\infty^{I'}$).
Also, for any map $f$ and point $x$, let $\mathcal{R}(f)$
denote the range of $f$ and $f'(x)$ (or $(f(x))'$) the derivative $\frac{df(x)}{dx}$.  
For $M, M^*, \delta\ge 0$, let
%\begin{eqnarray*}
%	\mathcal{Q}_I(M):=\{f\in \mathcal{C}_b(\mathbb{R}): \mathcal{R}(f)=I~\text{and}~|f(x)-f(y)|\le M|x-y|, \forall x, y \in I\}
%\end{eqnarray*}
%and
\begin{align*}
\mathcal{G}_{J}(\delta,M, M^*):=&\{g\in \mathcal{C}^1(\mathbb{R}_+): \mathcal{R}(g)=J, g(c)=c, g(d)=d,~\text{and}~\eqref{9}, \eqref{10}, \eqref{12}~\text{hold} \},\\
\mathcal{F}_{I}(\delta,M, M^*):=&\{f\in \mathcal{C}_b^1(\mathbb{R}): \mathcal{R}(f)=I, f(a)=a, f(b)=b,~ \text{and}~\eqref{2}, \eqref{4}, \eqref{06}~\text{hold} \},\\
\mathcal{B}_{J}(\delta,M,M^*):=&\{g\in \mathcal{C}^1(\mathbb{R}_+): \mathcal{R}(g)=J, g(c)=c, g(d)=d,~\text{and}~\eqref{11}, \eqref{12}~\text{hold} \}, \\
\mathcal{A}_{I}(\delta,M, M^*):=&\{f\in \mathcal{C}^1_b(\mathbb{R}): \mathcal{R}(f)=I, f(a)=a, f(b)=b,~\text{and}~\eqref{05}, \eqref{06}~\text{hold}\},
\end{align*}
%and
%\begin{eqnarray*}
%	\mathcal{B}_{J}(\delta,M,M^*)&:=&\left\{g\in \mathcal{C}^1_b(\mathbb{R}_+): \mathcal{R}(g)=J, g(c)=c, g(d)=d,\right. \\
%				& &~~~~~~~~~~~~~~~~~~~~~~~~~~~\text{and}~\eqref{11}, \eqref{12}~\text{hold} \},
%\end{eqnarray*}
where 
\begin{equation}\label{9}
\delta \le \frac{xg'(x)}{g(x)}\le M,\quad \forall x\in J,
\end{equation}
\begin{equation}\label{10}
\left|\frac{xg'(x)}{g(x)}\right|\le M,\quad \forall x\in \mathbb{R}_+\setminus J,
\end{equation}
\begin{equation}\label{11}
0 \le \frac{xg'(x)}{g(x)}\le M,\quad \forall x\in  J,
\end{equation}
\begin{equation}\label{12}
\left|\frac{xg'(x)}{g(x)}-\frac{yg'(y)}{g(y)}\right|\le M^*\left|\log \left(\frac{x}{y}\right) \right|, \quad \forall x, y \in J,
\end{equation}
\begin{equation}\label{2}
\delta \le f'(x)\le M,~\quad \forall x\in I,
\end{equation}
\begin{equation}\label{4}
|f'(x)|\le M,~\quad \forall x\in \mathbb{R}\setminus I,
\end{equation}
\begin{equation}\label{05}
0 \le f'(x)\le M,~\quad \forall x\in I,
\end{equation}
\begin{equation}\label{06}
|f'(x)-f'(y)|\le M^*|x-y|,~\quad \forall x, y\in I.
\end{equation}
Then it can be seen that 
$\mathcal{G}_{J}(\delta,M, M^*) \subseteq \mathcal{G}_{J}(\delta_1,M_1, M_1^*)$, $ 	\mathcal{F}_{I}(\delta,M, M^*)\subseteq \mathcal{F}_{I}(\delta_1,M_1, M_1^*)$, 
$	\mathcal{B}_{J}(\delta,M, M^*)\subseteq 	\mathcal{B}_{J}(\delta_1,M_1, M_1^*)$  and	$\mathcal{A}_{I}(\delta,M, M^*) \subseteq 	\mathcal{A}_{I}(\delta_1,M_1, M_1^*)$ whenever $0\le \delta_1\le \delta$, $0\le M\le  M_1$ and $0\le M^*\le M_1^*$. 

\begin{proposition}\label{P1}
	The following assertions are true for $M,M^*, \delta\ge 0$.  
	\begin{description}
		\item[(i)] $g\in 	\mathcal{G}_{J}(\delta,M, M^*)$ if and only if $f=\psi^{-1}\circ g\circ\psi  \in 	\mathcal{F}_{I}(\delta,M, M^*)$, where $\psi (x)=e^x$ and $I=\log J:=\{\log x:x\in J\}$. \label{P1i}
		\item[(ii)]
		If  $f \in 	\mathcal{A}_{I}(\delta,M, M^*)$, then  $g=\psi\circ f\circ \psi^{-1}\in 	\mathcal{B}_{J}(\delta,M, M^*)$, where $\psi (x)=e^x$ and $J=e^I:=\{e^x:x\in I\}$. \label{P1ii}
	\end{description}
\end{proposition}
\begin{proof}
	Given	$g\in 	\mathcal{G}_{J}(\delta,M, M^*)$, let $a:=\log c$ and $b:=\log d$. Then we obtain the interval $I=[a,b]$ with $a<b$, which satisfies $I=\log J$.
	Clearly, $f:=\psi^{-1}\circ g\circ\psi  \in \mathcal{C}^1(\mathbb{R})$. Also, we have $f(a)=\log g(e^a)=\log g(c)=\log c=a$, and similarly $f(b)=b$. So, $I\subseteq \mathcal{R}(f)$. The reverse inclusion follows by the definitions of $f$ and $I$, because $\mathcal{R}(g)=J$. Therefore $\mathcal{R}(f)=I$.
	
	Now, let $x\in I$ and $y\in \mathbb{R}\setminus I$.  Then there exist $u \in J$ and $v\in \mathbb{R}_+\setminus J$ such that $x=\log u$ and $y=\log v$. So, by \eqref{9} and \eqref{10}, we have 
	\begin{eqnarray*}
		\delta \le \frac{ug'(u)}{g(u)}\le M\quad\text{and}\quad \left| \frac{vg'(v)}{g(v)}\right|\le M,
	\end{eqnarray*}
	implying that 
	\begin{eqnarray*}
		\delta \le \frac{e^xg'(e^x)}{g(e^x)}\le M\quad\text{and}\quad \left| \frac{e^yg'(e^y)}{g(e^y)}\right|\le M,
	\end{eqnarray*}
	respectively. i.e., $\delta \le f'(x)\le M$ and $|f'(y)|\le M$, proving that $f$ satisfy \eqref{2} and \eqref{4}, respectively. 
	
	Next, let $x, y\in I$. Then there exist $u,v \in J$ such that $x=\log u$ and $y=\log v$. So, by \eqref{12}, we have
	\begin{eqnarray*}
		\left|\frac{ug'(u)}{g(u)}-\frac{vg'(v)}{g(v)}\right|\le M^*\left|\log \left(\frac{u}{v}\right) \right|, 
	\end{eqnarray*}
	implying that
	\begin{eqnarray*}
		\left|\frac{e^xg'(e^x)}{g(e^x)}-\frac{e^yg'(e^y)}{g(e^y)}\right|\le M^*\left|\log \left(\frac{e^x}{e^y}\right) \right|. 
	\end{eqnarray*}
	i.e., $|f'(x)-f'(y)|\le M^*|x-y|$, proving that $f$ satisfies \eqref{06}. Therefore $f \in 	\mathcal{F}_{I}(\delta,M, M^*)$. The proof of converse and that of result {\bf (ii)} are similar.
\end{proof}

\begin{proposition}\label{P2}
	If $M<1$ or $\delta >1$, then $\mathcal{G}_{J}(\delta,M, M^*)=\emptyset$.
	If $M=1$ or $\delta=1$, then $\mathcal{G}_{J}(\delta,M, M^*)=\{g\in \mathcal{C}^1(\mathbb{R}_+): g|_{J}={\rm id}\}$.
\end{proposition}

\begin{proof}
	Let $g\in \mathcal{G}_{J}(\delta,M, M^*)$.	Then by result {\bf (i)} of Proposition \ref{P1},  $f:=\psi^{-1}\circ g\circ\psi  \in 	\mathcal{F}_{I}(\delta,M, M^*)$, where $\psi (x)=e^x$ and $I=\log J$. So,  by using \eqref{2} we get 
	\begin{eqnarray}\label{5}
	f(x)-f(y)\le M(x-y), \quad \forall x,y \in I~\text{with}~x\ge y.
	\end{eqnarray}
	If $M<1$, then by setting $y=a$ in \eqref{5}, we have $f(x)<x$ for all $x\in I$ with $x>a$. This is a contradiction to the fact that $f(b)=b$, because $b>a$. So, $\mathcal{F}_{I}(\delta,M, M^*)=\emptyset$, and hence $\mathcal{G}_{J}(\delta,M, M^*)=\emptyset$ whenever $M<1$. A similar argument holds when $\delta>1$.
	
	If $M=1$, then from \eqref{5} we have
	\begin{eqnarray}\label{6}
	f(x)-f(y)\le x-y, \quad \forall x,y \in I~\text{with}~x\ge y.
	\end{eqnarray}
	For $x=b$, \eqref{6} implies that $f(y)\ge y$  for all $y\in I$ with $y<b$. Moreover, setting $y=a$ in \eqref{6}, we have $f(x)\le x$ for all $x\in I$ with $x> a$. Thus $f(x)=x$ for all $x\in I$, and therefore $f|_I={\rm id}$. This implies that $g|_J={\rm id}$.  The reverse inclusion is trivial. So, $\mathcal{G}_{J}(\delta,M, M^*)=\{g\in \mathcal{C}^1(\mathbb{R}_+): g|_{J}={\rm id}\}$. A similar argument holds when $\delta=1$.
\end{proof}

In view of the above proposition, we cannot seek solutions of \eqref{1} without imposing conditions on $M$ and $\delta$. So, henceforth we assume that $0<\delta\le 1\le M$ and $M^*>0 $.

\begin{proposition}\label{Fdm complete}
	The set $\mathcal{A}_{I}(\delta,M, M^*)$ is a complete metric space under the metric induced by $\|\cdot \|_{\mathcal{C}^1}$.
\end{proposition}
\begin{proof}
	It is easy to see that $\mathcal{A}_{I}(\delta,M, M^*)$ is a closed subset of $\mathcal{C}^1_b(\mathbb{R})$. Therefore, to prove the result, it suffices to show that   $\mathcal{C}^1_b(\mathbb{R})$ is complete with respect to the metric induced by $\|\cdot \|_{\mathcal{C}^1}$. So, consider an arbitrary Cauchy sequence  $(f_k)_{k\ge 1}$  in the normed linear space $(\mathcal{C}^1_b(\mathbb{R}),\|\cdot \|_{\mathcal{C}^1})$. Then, by the definition of $\|\cdot \|_{\mathcal{C}^1}$, it follows that $(f_k)_{k\ge 1}$ and $(f_k')_{k\ge 1}$ are Cauchy sequences in  $(\mathcal{C}_b(\mathbb{R}),\|\cdot \|_\infty)$, which is a Banach space. So, there exist $f, f_0\in \mathcal{C}_b(\mathbb{R})$ such that $f_k \to f$ and $f_k'\to f_0$ uniformly on $\mathbb{R}$ as $k\to \infty$. Also, since $(f_k')_{k\ge 1}$  is a bounded sequence in $(\mathcal{C}_b(\mathbb{R}),\|\cdot \|_\infty)$, there exists $\kappa >0$ such that $\|f_k'\|_\infty <\kappa$ for all $k\in \mathbb{N}$, implying that $|f_0(x)|\le \kappa$ for all $x\in \mathbb{R}$. We now claim that $f'=f_0$ on $\mathbb{R}$.
	
	Consider an arbitrary $x\in \mathbb{R}$. Let $I_x:=[u_x,v_x]$, where $u_x, v_x\in \mathbb{R}$ are chosen such that $u_x<x<v_x$. Since $f_k\to f$ uniformly on $I_x$, by using fundamental theorem of calculus, we have 
	\begin{eqnarray*}
		\int_{u_x}^{t}f'_k(y) dy=f_k(t)-f_k(u_x)\to f(t)-f(u_x)
	\end{eqnarray*}
for all $t\in I_x$. Also, since $f'_k\to f_0$ uniformly on $I_x$, we have
	\begin{eqnarray*}
	\int_{u_x}^{t}f'_k(y) dy \to \int_{u_x}^{t}f_0(y) dy
\end{eqnarray*}
	for all $t\in I_x$.
	Therefore 
		\begin{eqnarray*}
	 \int_{u_x}^{t}f_0(y) dy=f(t)-f(u_x),
	\end{eqnarray*}
	implying by fundamental theorem of calculus that $f'(t)=f_0(t)$ for all $t\in I_x$. 	
%	 Define a map $f_{x}$ on $I_{x}:=[u,v]$ by 
%	\begin{eqnarray*}
%		f_{x}(t)=\int_{u}^{t}f_0(t) dt, \quad \forall t\in I_x, 
%	\end{eqnarray*}
%	where $u, v\in \mathbb{R}$ are chosen such that $u<x<v$. Since $f_0$ is a bounded map on $\mathbb{R}$, clearly $f_{x}$ is well defined. Also, since $f_0$ is continuous on $I_{x}$, by fundamental theorem of calculus, it follows that $f_{x}'(t)=f_0(t)$ for all $t\in I_{x}$. Further, 
%	\begin{eqnarray*}
%		\|f_k-f_{x}\|^{I_{x}}_\infty&=&\displaystyle \sup_{t\in I_{x}}|f_k(t)-f_x(t)|\\
%		&=&\displaystyle \sup_{t\in I_{x}}\left|\int_{u}^{t}f_k'(t)dt-\int_{u}^{t}f_0(t)dt\right|\\
%		&\le &\displaystyle \sup_{t\in I_{x}}\int_{u}^{t}|f_k'(t)-f_0(t)|dt\\
%		&\le &\displaystyle \sup_{t\in I_{x}}\|f_k'-f_0\|_\infty^{[u,t]}(t-u)\\
%		&\le & (v-u)\|f_k'-f_0\|_\infty^{I_x}\\
%		&\le &  (v-u)\|f_k'-f_0\|_\infty, 
%	\end{eqnarray*}
%	implying that $f_k \to f_x$ uniformly on $I_x$, since $\|f_k'-f_0\|_\infty \to 0$ as $k\to \infty$. Already, $f_k \to f$ uniformly on $\mathbb{R}$ as $k\to \infty$. Therefore $f_x=f$ on $I_x$, and hence $f_x'(x)=f'(x)$. 
	Thus the claim holds and it follows that $f_k\to f$ in $(\mathcal{C}^1_b(\mathbb{R}),\|\cdot \|_{\mathcal{C}^1})$. The proof is completed.
\end{proof}

 In addition to the above proposition,  we need the following seven technical lemmas to prove our main result.

\begin{lemma}{\rm (Lemma 2.1 of \cite{zhang1990})} \label{L1}
	If $f\in \mathcal{C}^1(I)$ satisfy \eqref{05} and \eqref{06}, then 
	\begin{eqnarray*}
		|(f^k(x))'-(f^k(y))'|\le M^*\left(\sum_{j=k-1}^{2k-2}M^j\right)|x-y|, \quad \forall x, y \in I~\text{and}~k\ge 1.
	\end{eqnarray*}
\end{lemma}

\begin{lemma}{\rm (Lemma 2.2 of \cite{zhang1990})} \label{L0}
	Let $f_1, f_2\in \mathcal{C}(I)$ satisfy $|f_1(x)-f_1(y)|\le M|x-y|$ and $|f_2(x)-f_2(y)|\le M|x-y|$ for all $x, y \in I$. Then
	\begin{eqnarray}
	\|f_1^k-f_2^k\|_\infty^I\le \left(\sum_{j=0}^{k-1}M^j\right)\|f_1-f_2\|_\infty^I~~\text{for}~k=1,2,\ldots.
	\end{eqnarray}
\end{lemma}

\begin{lemma}{\rm (Lemma 2.3 of \cite{zhang1990})} \label{L4}
	Let $f_1, f_2\in \mathcal{C}^1(I)$ satisfy \eqref{05} and \eqref{06}. Then 
	\begin{eqnarray*}
		\|(f_1^{k+1})'-(f_2^{k+1})'\|_\infty^I\le (k+1)M^k\|f_1'-f_2'\|_\infty^I +Q(k+1)M^*\left(\sum_{j=1}^{k}(k-j+1)M^{k+j-1}\right)\|f_1-f_2\|_\infty^I
	\end{eqnarray*}
	for $k=0,1,2, \ldots$, where 
	\begin{eqnarray}\label{Q}
	Q(s):= \left\{\begin{array}{cll}
	0&\text{if}& s=1,\\
	1&\text{if}& s=2,3,\ldots.
	\end{array}\right.
	\end{eqnarray}
\end{lemma}

\begin{lemma}{\rm (Lemma 2.4 of \cite{zhang1990})} \label{L2}
	Let $f\in \mathcal{C}^1(I)$ satisfies \eqref{06} and let $\delta \le f'(x)$ for all $x\in I$. Then 
	\begin{eqnarray*}
		|(f^{-1}(x))'-(f^{-1}(y))'|&\le & \frac{M^*}{\delta^3}|x-y|, \quad \forall x, y \in I.
	\end{eqnarray*}
\end{lemma}

\begin{lemma}{\rm (Lemma 2.5 of \cite{zhang1990})} \label{L3}
	Let $f_1, f_2$ be homeomorphisms of $I$ onto itself such that 
	\begin{eqnarray*}
		|f_j(x)-f_j(y)|\le M^*|x-y|, ~\forall x, y \in I~\text{and}~j=1,2.
	\end{eqnarray*}
	Then 
	\begin{eqnarray*}
		\|f_1-f_2\|_\infty^I \le M^*\|f_1^{-1}-f_2^{-1}\|_\infty^I. 
	\end{eqnarray*}
\end{lemma}
For $\lambda_k\in [0,1]$, $1\le k \le n$ such that   $\sum_{k=1}^{n}\lambda_k=1$ and $f\in  \mathcal{F}_{I}(\delta,M, M^*)$, define $L_f: I\to I$ by
$$L_f(x)=\sum_{k=1}^{n}\lambda_kf^{k-1}(x),\quad x\in I.$$

\begin{lemma}\label{Lf(x) estimate}
	Let	$f\in  \mathcal{F}_{I}(\delta,M, M^*)$ and 
	$\lambda_k\in [0,1]$ for $1\le k \le n$ such that  $\sum_{k=1}^{n}\lambda_k=1$. Then 
	\begin{description}
		\item[(i)]  $L_f(a)=a$, $L_f(b)=b$ and $\mathcal{R}(L_f)=I$, \label{Lf(x) estimate1}
		\item[(ii)] $\lambda_1 \le L'_f(x) \le K_1, \quad \forall x\in I$,  \label{Lf(x) estimate2}
		\item[(iii)] $0<\dfrac{1}{K_1}\le (L^{-1}_f(x))'\le \dfrac{1}{\lambda_1}, \quad \forall x\in I$,  \label{Lf(x) estimate3} 
		\item[(iv)] $|L_{f}^{-1}(x)-L_f^{-1}(y)|\le \dfrac{1}{\lambda_1}|x-y|,\quad \forall x,y \in I$, \label{Lf(x) estimate3.1}
		\item[(v)] $|L'_f(x)-L'_f(y)|\le M^*K_0|x-y|,\quad \forall x,y \in I$,  \label{Lf(x) estimate4}
		\item[(vi)] $|(L^{-1}_f(x))'-(L^{-1}_f(y))'|\le \dfrac{M^*K_0}{\lambda_1^3}|x-y|,\quad \forall x,y \in I$,  \label{Lf(x) estimate5}
	\end{description}
	where
	\begin{eqnarray}\label{K0}
	K_0=\sum_{k=1}^{n-1}\lambda_{k+1}\left(
	\sum_{j=k-1}^{2k-2}M^j\right)\quad \text{and}\quad K_1=\sum_{k=1}^{n}\lambda_kM^{k-1}.
	\end{eqnarray} %$$K_0=\sum_{k=1}^{n-1}\lambda_{k+1}\left(
	%\sum_{j=k-1}^{2k-2}M^j\right)\quad \text{and}\quad K_1=\sum_{k=1}^{n}\lambda_kM^{k-1}.$$
\end{lemma}
\begin{proof}
	Since $0\le \lambda_k\le1$ for $1\le k\le n$ such that  $\sum_{k=1}^{n}\lambda_k=1$, it can be easily seen that $L_f(a)=a$, $L_f(b)=b$ and $a\le L_f(x)\le b$  for all $x\in I$. So, $\mathcal{R}(L_f)=I$, proving result {\bf (i)}.
	
	For each $x\in I$, by using \eqref{2}, we have
	\begin{eqnarray*}
		L_f'(x)=\lambda_1+\sum_{k=2}^{n}\lambda_k\left(\prod_{j=0}^{k-2}f'(f^j(x))\right)\le \lambda_1+\sum_{k=2}^{n}\lambda_kM^{k-1}=\sum_{k=1}^{n}\lambda_kM^{k-1}=K_1
	\end{eqnarray*}
	and 
	\begin{eqnarray*}
		L_f'(x)=\lambda_1+\sum_{k=2}^{n}\lambda_k\left(\prod_{j=0}^{k-2}f'(f^j(x))\right)\ge \lambda_1+\sum_{k=2}^{n}\lambda_k\delta^{k-1}\ge \lambda_1.
	\end{eqnarray*}
	This proves result {\bf (ii)}. 
	
	Since $L_f$ is strictly increasing map of $I$ onto itself, clearly $L_f^{-1}$ is a well defined map on $I$ such that $L_f^{-1}(a)=a$ and $L_f^{-1}(b)=b$. Also, as $L_f'(x)>0$ for all $x\in I$ by result {\bf (ii)}, using inverse function theorem it follows that $L_f^{-1}$ is differentiable on $I$ and  $(L_f^{-1})'(x)=(L_f(x))^{-1}$. Therefore, by result {\bf (ii)}, we have 
	\begin{eqnarray*}
		0<\dfrac{1}{K_1}\le (L_f(x))^{-1}\le \dfrac{1}{\lambda_1},
	\end{eqnarray*}
	proving result {\bf (iii)}. Result {\bf (iv)} follows from result {\bf (iii)} and the mean value theorem. 
	
	Since $f$ satisfy \eqref{2} and \eqref{06}, by using Lemma \ref{L1}, we have  
	\begin{eqnarray*}
		|L'_f(x)-L'_f(y)|&=&\left| \sum_{k=1}^{n-1}\lambda_{k+1}(f^k(x))'-\sum_{k=1}^{n-1}\lambda_{k+1}(f^k(y))'\right|\\
		&\le &\sum_{k=1}^{n-1}\lambda_{k+1}|(f^k(x))'-(f^k(y))'|\\
		&\le &\sum_{k=1}^{n-1}\lambda_{k+1}M^*\left(\sum_{j=k-1}^{2k-2}M^j\right)|x-y|=M^*K_0|x-y| 
	\end{eqnarray*}
	for all $x, y \in I$, proving result {\bf (v)}.
	Also, since results {\bf (ii)} and {\bf (iv)} are true for $L_f$, result {\bf (vi)} follows by Lemma \ref{L2}.   
\end{proof}

\begin{lemma}\label{Lf estimate}
	Let	$f_1, f_2\in  \mathcal{F}_{I}(\delta,M, M^*)$ and 
	$\lambda_k\in [0,1]$ for $1\le k \le n$ such that  $\sum_{k=1}^{n}\lambda_k=1$. Then 
	\begin{description}
		\item[(i)] $\|L_{f_1}-L_{f_2}\|_\infty^I \le K_2\|f_1-f_2\|_\infty^I$, \label{Lf estimate1}
		\item[(ii)] $\|L^{-1}_{f_1}-L^{-1}_{f_2}\|_\infty^I \le \dfrac{1}{\lambda_1}\|L_{f_1}-L_{f_2}\|_\infty^I$, \label{Lf estimate2}
		\item[(iii)] $\|L^{-1}_{f_1}-L^{-1}_{f_2}\|_\infty^I \le \dfrac{K_2}{\lambda_1}\|f_1-f_2\|_\infty^I$, \label{Lf estimate3}
		\item[(iv)] $\|L'_{f_1}-L'_{f_2}\|_\infty^I \le K_3\|f_1-f_2\|_\infty^I+K_4\|f'_1-f'_2\|_\infty^I$, \label{Lf estimate4}
		\item[(v)] $\|(L^{-1}_{f_1})'-(L^{-1}_{f_2})'\|_\infty^I \le K_5\|f_1-f_2\|_\infty^I+K_6\|f'_1-f'_2\|_\infty^I$, \label{Lf estimate5}
	\end{description}
	where 
	\begin{eqnarray}
	K_2&=&\sum_{k=2}^{n}\lambda_k\left(
	\sum_{j=0}^{k-2}M^{j}\right),\label{K2}\\  K_3&=&\sum_{k=0}^{n-2}\lambda_{k+2}Q(k+1)M^*\left(\sum_{j=1}^{k}(k-j+1)M^{k+j-1}\right), \nonumber\\
	K_5&=&\dfrac{K_3}{\lambda_1^2}+\dfrac{M^*K_0K_2}{\lambda_1^3},\nonumber \\
	K_4&=&\sum_{k=0}^{n-2}\lambda_{k+2}(k+1)M^k, \quad \text{and}\quad K_6=\dfrac{K_4}{\lambda_1^2} \label{K46}
	\end{eqnarray}
	with $Q$ defined as in \eqref{Q}.
\end{lemma}
\begin{proof}
	For each $x\in I$, we have 
	\begin{eqnarray*}
		|L_{f_1}(x)-L_{f_2}(x)|&\le & \sum_{k=2}^{n}\lambda_k|f_1^{k-1}(x)-f_2^{k-1}(x)|\\
		&\le &\sum_{k=2}^{n}\lambda_k\|f_1^{k-1}-f_2^{k-1}\|_\infty^I\\
		&\le& \sum_{k=2}^{n}\lambda_k\left(\sum_{j=0}^{k-2}M^j\right)\|f_1-f_2\|_\infty^I~~(\text{by using Lemma \ref{L0}})\\
		&=&K_2\|f_1-f_2\|_\infty^I,
	\end{eqnarray*}
	proving result {\bf (i)}. By using result {\bf (iv)} of Lemma \ref{Lf(x) estimate} for $f_1, f_2$ and Lemma \ref{L3} for $L_{f_1}^{-1}, L_{f_2}^{-1}$, we get that
	\begin{eqnarray*}
		\|L^{-1}_{f_1}-L^{-1}_{f_2}\|_\infty^I &\le& \dfrac{1}{\lambda_1}\|L_{f_1}-L_{f_2}\|_\infty^I\\
		&\le& \dfrac{K_2}{\lambda_1}\|f_1-f_2\|_\infty^I~~~(\text{by using result {\bf (i)}}),
	\end{eqnarray*}
	proving results {\bf (ii)} and {\bf (iii)}.  For each $x\in I$, by using Lemma \ref{L4} we have 
%	$	|L'_{f_1}(x)-L'_{f_2}(x)|$
	\begin{eqnarray*}
			|L'_{f_1}(x)-L'_{f_2}(x)|
	&\le&\left| \sum_{k=2}^{n}\lambda_{k}(f_1^{k-1}(x))'-\sum_{k=2}^{n}\lambda_{k}(f_2^{k-1}(x))'\right|\\
	&\le &\sum_{k=2}^{n}\lambda_{k}|(f_1^{k-1}(x))'-(f_2^{k-1}(x))'|\\
		&\le & \sum_{k=2}^{n}\lambda_{k}\|(f_1^{k-1})'-(f_2^{k-1})'\|_\infty^I
	\end{eqnarray*}
	\begin{eqnarray*}
		&=&\sum_{k=0}^{n-2}\lambda_{k+2}\|(f_1^{k+1})'-(f_2^{k+1})'\|_\infty^I \\
		&\le &\sum_{k=0}^{n-2}\lambda_{k+2}\left\{(k+1)M^k\|f_1'-f_2'\|_\infty^I\right. \big.+Q(k+1)M^*\left(\sum_{j=1}^{k}(k-j+1)M^{k+j-1}\right)\|f_1-f_2\|_\infty^I\big\}\\
	%	&&~~~~~~~~~~~~~~~~~~~~~~~~~~~~~~~~~~~~~~~~~~~(\text{using Lemma \ref{L4}})\\
		&=&K_3\|f_1-f_2\|_\infty^I+K_4\|f'_1-f'_2\|_\infty^I,
	\end{eqnarray*}
	and therefore result {\bf (iv)} is proved. Also, for each $x\in I$, we have 
	
	\noindent $|(L^{-1}_{f_1}(x))'-(L^{-1}_{f_2}(x))'|$
	\begin{eqnarray*}
		%|(L^{-1}_{f_1}(x))'-(L^{-1}_{f_2}(x))'|
		&=&\left|\frac{1}{L_{f_1}'(L_{f_1}^{-1}(x))}-\frac{1}{L_{f_2}'(L_{f_2}^{-1}(x))} \right|\\
		&\le& \frac{1}{\lambda_1^2}|L_{f_1}'(L_{f_1}^{-1}(x))-L_{f_2}'(L_{f_2}^{-1}(x))|~~(\text{using result {\bf (ii)} of Lemma}~\ref{Lf(x) estimate})\\
		&\le & \frac{1}{\lambda_1^2}\left\{|L_{f_1}'(L_{f_1}^{-1}(x))-L_{f_2}'(L_{f_1}^{-1}(x))|+|L_{f_2}'(L_{f_1}^{-1}(x))-L_{f_2}'(L_{f_2}^{-1}(x))|\right\}\\
		&\le & \frac{1}{\lambda_1^2}\left\{\|L_{f_1}'-L_{f_2}'\|_\infty^I+M^*K_0|L_{f_1}^{-1}(x)-L_{f_2}^{-1}(x)|\right\}~~(\text{using result {\bf (v)} of Lemma}~\ref{Lf(x) estimate})\\
		&\le & \frac{1}{\lambda_1^2}\left\{K_3\|f_1-f_2\|_\infty^I+K_4\|f'_1-f'_2\|_\infty^I+M^*K_0\|L_{f_1}^{-1}-L_{f_2}^{-1}\|_\infty^I\right\}~~(\text{using result {\bf (iv)}})\\
		&\le & \frac{1}{\lambda_1^2}\left\{K_3\|f_1-f_2\|_\infty^I+K_4\|f'_1-f'_2\|_\infty^I+\frac{M^*K_0K_2}{\lambda_1}\|f_1-f_2\|_\infty^I\right\}~~(\text{using result {\bf (iii)}})\\
		&=&K_5\|f_1-f_2\|_\infty^I+K_6\|f'_1-f'_2\|_\infty^I,
	\end{eqnarray*}
	proving result {\bf (v)}.  
\end{proof}

\section{The  result}\label{Sec3}
We now prove our intended result on the existence and uniqueness of differentiable solutions of \eqref{1}.

\begin{theorem}\label{existance}
	Let	$\lambda_1>0$ and $\lambda_k\in [0,1]$ for $2\le k \le n$ such that  $\sum_{k=1}^{n}\lambda_k=1$. Let $K_0, K_2, K_4$ and $K_6$ be as defined in Lemmas \ref{Lf(x) estimate} and \ref{Lf estimate}. Then for arbitrary $G\in  \mathcal{G}_{J}(\delta,\lambda_1M, M^*)$ there exists a unique solution $g$ of \eqref{1} in $\mathcal{B}_{J}(\delta,M, M')$ provided $\lambda_1>K_0M^2>0$ and $0<K<1$, where
	\begin{eqnarray}
	M'=\dfrac{M^*}{\lambda_1-K_0M^2}\quad \text{and}\quad K=\max \left\{\dfrac{K_2}{\lambda_1}+\lambda_1MK_5',~ \lambda_1MK_6\right\}
	\end{eqnarray}
	with
	\begin{eqnarray}
	K_5'&=&\dfrac{K_3'}{\lambda_1^2}+\dfrac{M'K_0K_2}{\lambda_1^3},\\
	K_3'&=&\sum_{k=0}^{n-2}\lambda_{k+2}Q(k+1)M'\left(\sum_{j=1}^{k}(k-j+1)M^{k+j-1}\right). \label{K5'}
	\end{eqnarray}
	
	%	$$M'=\dfrac{M^*}{\lambda_1-K_0M^2}, \quad K=\max \left\{\dfrac{K_2}{\lambda_1}+\lambda_1MK_5',~ \lambda_1MK_6\right\}, \quad 	K_5'=\dfrac{K_3'}{\lambda_1^2}+\dfrac{M'K_0K_2}{\lambda_1^3}$$
	%	and 
	%	$$K_3'=\sum_{k=0}^{n-2}\lambda_{k+2}Q(k+1)M'\left(\sum_{j=1}^{k}(k-j+1)M^{k+j-1}\right).$$
\end{theorem}

\begin{proof}
	Let $G\in  \mathcal{G}_{J}(\delta,\lambda_1M, M^*)$. Then by result {\bf (i)} of Proposition \ref{P1}, we have $F=\psi^{-1}\circ G\circ \psi \in  \mathcal{F}_{I}(\delta,\lambda_1M, M^*)$, where  $\psi (x)=e^x$ and $I=\log J$. Let $\log c=a$ and $\log d=b$. Then $I=[a,b]$ such that $a<b$.

	Define $T:\mathcal{A}_{I}(\delta,M,\frac{M^*}{\lambda_1-K_0M^2})\to \mathcal{C}^1_b(\mathbb{R})$ by
	$$Tf(x)=L_f^{-1}(F(x)),\quad x\in \mathbb{R}.$$
	By definitions of $F$ and $L_f$, we have $Tf(a)=a$ and $Tf(b)=b$. This implies that $I\subseteq \mathcal{R}(Tf)$. Also, since $L_f^{-1}:I\to I$, we have $\mathcal{R}(L_f^{-1})\subseteq I$, and therefore $\mathcal{R}(Tf)\subseteq I$. So, $\mathcal{R}(Tf)=I$. Further, from result {\bf (iii)} of Lemma \ref{Lf(x) estimate}, we have  
	\begin{eqnarray*}
		0<\frac{\delta}{K_1}\le (Tf)'(x)=(L_f^{-1})'(F(x))F'(x)\le \frac{1}{\lambda_1}\lambda_1M=M, \quad \forall x\in I.
	\end{eqnarray*}
	Moreover, since $F\in \mathcal{F}_{I}(\delta,\lambda_1M, M^*)$,  for each $x,y \in I$, we have 
	
	\noindent	$|(Tf)'(x)-(Tf)'(y)|$
	\begin{eqnarray*}
		%|(Tf)'(x)-(Tf)'(y)|
		&=&|(L_f^{-1})'(F(x))F'(x)-(L_f^{-1})'(F(y))F'(y)|\\
		&= & |(L_f^{-1})'(F(x))F'(x)-(L_f^{-1})'(F(x))F'(y)+(L_f^{-1})'(F(x))F'(y)-(L_f^{-1})'(F(y))F'(y)|\\
		&\le & |(L_f^{-1})'(F(x))||F'(x)-F'(y)|+|(L_f^{-1})'(F(x))-(L_f^{-1})'(F(y))||F'(y)|\\
		&\le & \frac{M^*}{\lambda_1}|x-y|+\frac{K_0MM'}{\lambda_1^2}|F(x)-F(y)|~~(\text{using results {\bf (iii)}}~\text{and {\bf (vi)} of Lemma}~\ref{Lf(x) estimate})\\
		&\le & \frac{M^*}{\lambda_1}|x-y|+\frac{K_0MM'}{\lambda_1^2}\lambda_1M|x-y|\\
		&=& \left(\frac{M^*}{\lambda_1}+\frac{K_0M^2M'}{\lambda_1}\right)|x-y|\\
		&=&\frac{M^*}{\lambda_1-K_0M^2}|x-y|.
	\end{eqnarray*}
	So, $Tf\in \mathcal{A}_{I}(\delta,M,\frac{M^*}{\lambda_1-K_0M^2})$, which proves that $T$ is a self-map on $\mathcal{A}_{I}(\delta,M,\frac{M^*}{\lambda_1-K_0M^2})$.
	In order to prove that $T$ is a contraction, consider any $f_1, f_2 \in \mathcal{A}_{I}(\delta,M,\frac{M^*}{\lambda_1-K_0M^2})$. Then for each $x\in \mathbb{R}$, we have
	\begin{eqnarray}
	|Tf_1(x)-Tf_2(x)|&=&|L_{f_1}^{-1}(F(x))-L_{f_2}^{-1}(F(x))|\nonumber\\
	&\le&\|L_{f_1}^{-1}-L_{f_2}^{-1}\|_\infty^I~~(\text{since}~F(x)\in I)\nonumber\\
	&\le &\frac{1}{\lambda_1}|L_{f_1}-L_{f_2}\|_\infty^I~~~(\text{using result {\bf (ii)} of Lemma}~\ref{Lf estimate})\nonumber\\
	&\le &\frac{K_2}{\lambda_1}\|f_1-f_2\|_\infty^I~~~(\text{using result {\bf (i)} of Lemma}~\ref{Lf estimate})\nonumber\\
	&\le &\frac{K_2}{\lambda_1}\|f_1-f_2\|_\infty, \nonumber
	\end{eqnarray}
	implying that 
	\begin{eqnarray}\label{s1}
	\|Tf_1-Tf_2\|_\infty\le \frac{K_2}{\lambda_1}\|f_1-f_2\|_\infty.
	\end{eqnarray}
	Also, for each $x\in \mathbb{R}$, we have
	
	\noindent $|(Tf_1)'(x)-(Tf_2)'(x)|$
	\begin{eqnarray*}
		%	|(Tf_1)'(x)-(Tf_2)'(x)|
		&=&|(L_{f_1}^{-1})'(F(x))F'(x)-(L_{f_2}^{-1})'(F(x))F'(x)|\\
		&= & |F'(x)| |(L_{f_1}^{-1})'(F(x))-(L_{f_2}^{-1})'(F(x))|\\
		&\le & \lambda_1M\| (L_{f_1}^{-1})'-(L_{f_2}^{-1})'\|_\infty^I~~~(\text{since}~F(x)\in I)\\
		&\le & \lambda_1MK_5'\|f_1-f_2\|_\infty^I+\lambda_1MK_6\|f'_1-f'_2\|_\infty^I~~(\text{using result {\bf (v)} of Lemma}~\ref{Lf estimate})\\
		&\le &\lambda_1MK_5'\|f_1-f_2\|_\infty+\lambda_1MK_6\|f'_1-f'_2\|_\infty.
	\end{eqnarray*}
	Therefore 
	\begin{eqnarray}\label{s2}
	\|(Tf_1)'-(Tf_2)'\|_\infty \le \lambda_1MK_5'\|f_1-f_2\|_\infty+\lambda_1MK_6\|f'_1-f'_2\|_\infty.
	\end{eqnarray}
	Thus 
	\begin{eqnarray*}
		\|Tf_1-Tf_2\|_{\mathcal{C}^1}&=&	\|Tf_1-Tf_2\|_\infty+	\|(Tf_1)'-(Tf_2)'\|_\infty \\
		&\le & \frac{K_2}{\lambda_1}\|f_1-f_2\|_\infty+\lambda_1MK_5'\|f_1-f_2\|_\infty+\lambda_1MK_6\|f'_1-f'_2\|_\infty\\
		& &~~~~~~~~~~~~~~~~~~~~~~~~~~~~~~~~~~~~~~~~(\text{using}~\eqref{s1}~\text{and}~\eqref{s2})\\
		&=& \left(\frac{K_2}{\lambda_1}+\lambda_1MK_5'\right)\|f_1-f_2\|_\infty+\lambda_1MK_6\|f'_1-f'_2\|_\infty\\
		&\le & K\|f_1-f_2\|_{\mathcal{C}^1}.
	\end{eqnarray*}
	Since $0<K<1$, it follows that $T$ is a contraction. By Proposition \ref{Fdm complete}, $\mathcal{A}_{I}(\delta,M,\frac{M^*}{\lambda_1-K_0M^2})$ is complete, and hence by Banach's contraction principle, $T$ has a unique fixed point in $\mathcal{A}_{I}(\delta,M,\frac{M^*}{\lambda_1-K_0M^2})$.  That is, there exists unique $f\in \mathcal{A}_{I}(\delta,M,\frac{M^*}{\lambda_1-K_0M^2})$ such that $L_f^{-1}(F(x))=f(x)$ for all $x\in \mathbb{R}$, which proves that $f$ is the unique solution of \eqref{(*)} in $\mathcal{A}_{I}(\delta,M,\frac{M^*}{\lambda_1-K_0M^2})$. This implies by Propositions \ref{P3} and result {\bf (ii)} of Proposition \ref{P1} that $g=\psi \circ f\circ\psi^{-1}$ is the unique solution of \eqref{1} in $\mathcal{B}_{J}(\delta,M,\frac{M^*}{\lambda_1-K_0M^2})$.
\end{proof}

\section{Example and Remarks}\label{Sec4}
The following example illustrates Theorem \ref{existance}.
\begin{example}
	{\rm 	Consider the equation
		\begin{eqnarray} \label{Ex4}
		(g(x))^{\lambda_1} (g^{2}(x))^{\lambda_2}=G(x), \quad x\in \mathbb{R}_+,
		\end{eqnarray}
		where $\lambda_1=\frac{9}{10}$, $\lambda_2=\frac{1}{10}$ and 
		\begin{equation*}
		G(x)=\left\{\begin{array}{cl}
		1& {\rm if}~x\in (0,1],\\
		2^{x-1}& {\rm if}~ x\in [1,2],\\
		2^{\frac{\log 2}{\log x}}& {\rm if}~x\in [2, \infty).
		\end{array}\right.
		\end{equation*}
		Let $f(x):=\log g(e^x)$ and $F(x):=\log G(e^x)$ for $x\in \mathbb{R}$. Then \eqref{Ex4} reduces to the polynomial-like equation $	\lambda_1	f(x)+\lambda_2f^2(x)=F(x)$,
%		\begin{eqnarray*}\label{Ex3}
%			\lambda_1	f(x)+\lambda_2f^2(x)=F(x),
%		\end{eqnarray*}
		where $F:\mathbb{R}\to \mathbb{R}$ is defined by
		\begin{equation*}
		F(x)=\left\{\begin{array}{cl}
		0& {\rm if}~x\le 0,\\
		(e^x-1)	\log 2& {\rm if}~ x\in [0,\log 2],\\
		\frac{(\log 2)^2}{x}& {\rm if}~x\ge \log 2.
		\end{array}\right.
		\end{equation*}
		Choose $\delta=\log 2$, $M=\frac{20\log 2}{9}$ and $M^*=2\log 2$. Then, clearly $F(0)=0$, $F(\log 2)=\log 2$, $F(x)\in I$ for all $x\in \mathbb{R}$, $\delta\le F'(x)\le \lambda_1M$ for all $x\in I$, $|F'(x)|\le \lambda_1M$ for all $x\in \mathbb{R}\setminus I$, and $|F''(x)|\le M^*$ for all $x\in I$, where $I=[0, \log 2]$. Therefore 
		$F\in \mathcal{F}_I(\delta, \lambda_1M, M^*)$, implying  by result {\bf (i)} of Proposition \ref{P1} that $G\in \mathcal{G}_J(\delta, \lambda_1M, M^*)$, where $J=[1,2]$.  Also, from \eqref{K0}-\eqref{K5'}, we have $K_0=K_2=K_4=\lambda_2$, $K_3'=0$, $K_5'=\frac{M^*\lambda_2^2}{\lambda_1^3(\lambda_1-\lambda_2M^2)}$ and $K_6=\frac{\lambda_2}{\lambda_1^2}$, and therefore 
		\begin{eqnarray*}
			K=\max \left\{\dfrac{\lambda_2}{\lambda_1}+\dfrac{MM^*\lambda_2^2}{\lambda_1^2(\lambda_1-\lambda_2M^2)},~ \frac{M\lambda_2}{\lambda_1}\right\}=\max \{0.15089,  0.17115\}=0.0.17115 \in  (0,1).
		\end{eqnarray*} 
		Further, $K_0M^2=\lambda_2M^2=0.23726 \in (0, \lambda_1)$. 
		Thus, all the hypotheses of Theorem \ref{existance} are satisfied. Hence \eqref{Ex4} has a unique solution $g$ in $\mathcal{G}_J(\delta, M, M')$, where $M'=\frac{M^*}{\lambda_1-\lambda_2M^2}=2.09176$. }
\end{example}

Since we have assumed in Theorem \ref{existance} that $\lambda_1>0$ for a technical reason, we cannot solve the iterative root problem $g^n=G$ on $\mathbb{R}_+$  using this result. Further, the assumption that $K>0$ made in Theorem \ref{existance} implies that $\lambda_k\ne 0$ for some $2\le k\le n$. 
However, if $\lambda_k=0$ for all $2\le k \le n$, then $g=G$ is the unique solution of \eqref{1} on $\mathbb{R}_+$. 
Additionally, the assumption that $\sum_{k=1}^{n}\lambda_k=1$ is not severe. 
In fact, if $\sum_{k=1}^{n}\lambda_k>1$, then we can divide each of the exponents in \eqref{1} by $\sum_{k=1}^{n}\lambda_k$ to get the normalized equation, 
but the assumptions on $G$ have to be modified suitably.

Moreover, by using the above observation, Theorem \ref{existance} and Proposition \ref{P4}, we can indeed solve \eqref{1} on $\mathbb{R}_-$ whenever $\lambda_k\in \mathbb{Z}$ for all $1\le k\le n$ such that $\sum_{k=1}^{n}\lambda_k$ is odd. On the other hand, if
$\alpha_k\in \mathbb{R}\setminus \mathbb{Z}$ for some $1\le k \le n$, then for any $G, g\in \mathcal{C}^1(\mathbb{R}_-)$,  $x \mapsto \prod_{k=1}^{n}(g^k(x))^{\lambda_k}$ is a  multi-valued complex map, whereas $x \mapsto G(x)$ is a  single valued real map. So, to obtain the equality \eqref{1},
% $\prod_{k=1}^{n}(g^k(x))^{\lambda_k}=G(x)$, 
we have to choose branches of the complex logarithm  suitably, which  depends both on $x$ and  each term of the product $\prod_{k=1}^{n}(g^k(x))^{\lambda_k}$.  Therefore,  solving \eqref{1} on $\mathbb{R}_-$ in this case is very difficult.

%%%%%%%%%%%%%%%%%%%%%%%%%%%%%%%%%

\bibliographystyle{amsplain}

\end{document}